\documentclass[preprint,11pt]{elsarticle}

\usepackage{float}
\usepackage{amstext,amsfonts,amsmath,amssymb,array,tikz}
\usepackage{amsthm,moresize,mathtools}
\usepackage{graphicx}
\makeatletter
\def\@author#1{\g@addto@macro\elsauthors{\normalsize%
		\def\baselinestretch{1}%
		\upshape\authorsep#1\unskip\textsuperscript{%
			\ifx\@fnmark\@empty\else\unskip\sep\@fnmark\let\sep=,\fi
			\ifx\@corref\@empty\else\unskip\sep\@corref\let\sep=,\fi
		}%
		\def\authorsep{\unskip,\speace}%
		\global\let\@fnmark\@empty
		\global\let\@corref\@empty  
		\global\let\sep\@empty}%
	\@eadauthor={#1}
}
\makeatother

\def\dj{d\kern-0.4em\char"16\kern-0.1em}

\newtheorem{theorem}{Theorem}[section]
\newtheorem{lemma}[theorem]{Lemma}

\newtheorem{corollary}[theorem]{Corollary}

\theoremstyle{definition}
\newtheorem{example}[theorem]{Example}
\newtheorem{remark}[theorem]{Remark}

\newcommand{\ba}{\begin{array}}
\newcommand{\ea}{\end{array}}

\newcommand{\bc}{\begin{center}}
\newcommand{\ec}{\end{center}}

\journal{the journal.}

\begin{document}

\begin{frontmatter}

\title{\bf Laplacian eigenvalues and eigenspaces of cographs generated by  finite sequence}

\author{Santanu Mandal}\corref{ca}
\ead{santanu.vumath@gmail.com}
\address{Department of Mathematics,
National Institute of Technology Rourkela,
Rourkela - 769008, India} \cortext[ca]{Corresponding author.}

\author{Ranjit Mehatari}
\ead{ranjitmehatari@gmail.com, mehatarir@nitrkl.ac.in}
\address{Department of Mathematics,
	National Institute of Technology Rourkela,
	Rourkela - 769008, India}

\author{Zoran Stani\' c}
\ead{zstanic@matf.bg.ac.rs}
\address{Faculty of Mathematics, University of Belgrade, 
	Studentski trg 16, 11 000 Belgrade, Serbia}

\begin{abstract} 
	In this paper we consider  particular  graphs defined by $\overline{\overline{\overline{K_{\alpha_1}}\cup K_{\alpha_2}}\cup\cdots \cup K_{\alpha_k}}$, where $k$ is even, $K_\alpha$ is a complete graph on $\alpha$ vertices, $\cup$ stands for the disjoint union and an overline denotes the complementary graph. These graphs do not contain the $4$-vertex path as an induced subgraph, i.e., they belong to the class of cographs. In addition, they are iteratively constructed from the generating sequence $(\alpha_1, \alpha_2, \ldots, \alpha_k)$. Our primary question is what invariants or graph properties can be deduced form a given sequence. In this context, we compute  the Lapacian eigenvalues and the corresponding eigenspaces, and derive a lower and an upper bound  for the number of distinct Laplacian eigenvalues. We also determine the graphs under consideration with a fixed number of vertices that either minimize or maximize the algebraic connectivity (that is the second smallest Laplacian eigenvalue). The   clique number is computed in terms of a generating sequence and a relationship between it and the algebraic connectivity is established.
\end{abstract}


\end{frontmatter}

\section{Introduction}

Throughout the paper, all graphs are assumed to be finite, undirected and without loops or multiple edges. Cographs were introduced in  1960's \cite{Kel}, and this class has been rediscovered independently by several authors in many
equivalent ways since then. They are intensively studied in the domain of structural considerations, spectral properties and applications. A short background is given in the next section. A \textit{cograph} is usually defined as a $P_4$-free graph, i.e., a graph that does not contain the 4-vertex path as an induced subgraph. It is also known that the class of cographs is closed under taking disjoint unions or complementation, and therefore an alternative definition says that an isolated vertex is a cograph, and if $G$ and $H$ are cographs, then their disjoint union $G\cup H$ is a cograph  and their join $\overline{\overline{G}\cup\overline{H}}$ is also a cograph \cite{Corneil 1}; as usual, an overline designates the complementary graph.  

\begin{figure}
	\includegraphics[width=\textwidth]{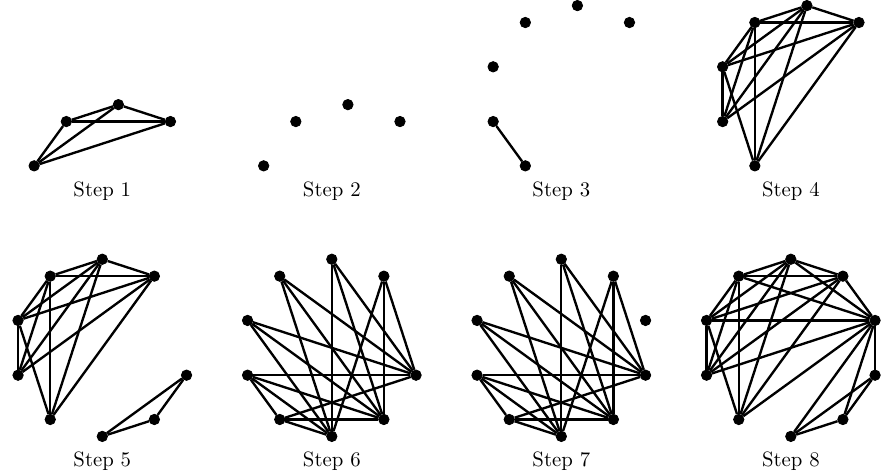}
	\caption{Construction of the $\mathcal{C}$-graph $C(4,2,3,1)$}
	\label{Cgraph_fig1}
\end{figure}  

Particular cographs considered in this study are defined in the following way (and the fact that they are cographs follows from the mentioned equivalent definitions). Let $K_\alpha$ denote the complete graph on $\alpha$ vertices. For positive integers  $\alpha_1,\alpha_2,\ldots,\alpha_k$,  $C(\alpha_1,\alpha_2,\ldots,\alpha_k)$ denotes the cograph defined recursively by
$$\left\{\begin{array}{l}
	C(\alpha_1)=\overline{K}_{\alpha_1}, \\	C(\alpha_1,\alpha_2,\ldots,\alpha_i)=\overline{C(\alpha_1,\alpha_2,\ldots,\alpha_{i-1})\cup K_{\alpha_i}},~ \text{for}~ 2\leq i\leq k.
\end{array}\right.$$
Simultaneously, $(\alpha_1, \alpha_2, \ldots, \alpha_k)$ is referred to as the \textit{generating sequence} of the corresponding cograph. In simple words, to construct  $C(\alpha_1,\alpha_2,\ldots,\alpha_k)$, we begin with $\overline{K}_{\alpha_1}$. In the next step, we take the disjoint union of $K_{\alpha_2}$ and the graph obtained in the first step, and then take the complementary graph. Proceeding in this way, we finally take the disjoint union of $K_{\alpha_k}$ and the graph obtained in the $(k-1)$th step, and finalize the construction by another complementation. Equivalently, $$C(\alpha_1,\alpha_2,\ldots,\alpha_k)\cong \overline{\overline{\overline{K_{\alpha_1}}\cup K_{\alpha_2}}\cup\cdots \cup K_{\alpha_k}}.$$  
 A construction of the 10-vertex cograph $C(4,2,3,1)$ is illustrated in  Fig.~\ref{Cgraph_fig1}. Let $\mathcal{C}$ denote the class of cographs constructed in  above way. For $G\in \mathcal{C}$, we simply say that $G$ is a \textit{$\mathcal{C}$-graph}. In particular, $\mathcal{C}_{even}$ denotes the $\mathcal{C}$-graphs that are generated by an even sequence. Accordingly, they are called \textit{$\mathcal{C}_{even}$-graphs}. In the entire paper our focus is on this particular class, so $k$ is assumed to be even.

To explain our motivation, we recall that a \textit{threshold graph} is a $\{P_4,2K_2, C_4\}$-free graph, i.e., a particular cograph without an induced subgraph isomorphic to either two parallel edges or the 4-vertex cycle. It is known that every threshold graph is generated by  a finite binary sequence \cite{Aguilar2,Andj,Bapat}. Moreover, the same holds for their bipartite counterparts known as chain graphs \cite{Ala,Mandal2}, not defined here. In this context, an experienced reader will surely recall the every $n$-vertex tree (even more, a labelled $n$-vertex tree) is generated by a unique sequence of $n-2$ numbers called the Pr\"{u}fer sequence~\cite{Pru}. And, of course, there are other graphs that are uniquely generated by a finite sequence in a similar way. This approach appears to be very convenient since the entire graph is fully determined by a simple sequence; instead of `a sequence', one may also say `a vector' or `a string'.  Moreover, a generating sequence provides information about many structural and spectral parameters.

In contrast to  the aforementioned graph classes, a representation of a $\mathcal{C}$-graph by a finite sequence may or may not be unique; in other words, it may occur that different sequences are associated with the same graph; for example, $C(1, 2, 2)$ and $C(1,1,1,2)$ are isomorphic. However, according to \cite{MM}, if $G\in \mathcal{C}_{even}$, then there is a unique even sequence $(\alpha_1, \alpha_2, \ldots, \alpha_{2k})$ such that $G\cong C(\alpha_1,\alpha_2,\ldots,\alpha_{2k})$. In this paper, we investigate the invariants that can be deduced form the generating sequence of a $\mathcal{C}_{even}$-graph. Here is the outline of the results established in the forthcoming sections. 

To give a clear insight into the class $\mathcal{C}$ and the subclass $\mathcal{C}_{even}$, we first give some data and comparisons with certain related graph classes.   

If $A$  is the standard adjacency matrix of a graph $G$ and $D$ is the diagonal matrix of vertex degrees, then $L=D-A$ is the \textit{Laplacian matrix} of $G$. Its eigenvalues, spectrum and eigenvectors are known as the \textit{Laplacian eigenvalues}, the \textit{Laplacian spectrum} and the \textit{Laplacian eigenvectors} of $G$. In particular, the second smallest Laplacian eigenvalue $a(G)$ is called the \textit{algebraic connectivity} of $G$. In this paper,  we establish a recurrence formula that computes the Laplacian eigenvalues and the Laplacian eigenvectors of a $\mathcal{C}$-graph in terms of its generating sequence. We also give a lower bound and an upper bound for the number of distinct Laplacian eigenvalues, and for each bound we construct $\mathcal{C}_{even}$-graphs that attain it. 

We consider  $\mathcal{C}_{even}$-graphs with a fixed number of vertices that either maximize or minimize the algebraic connectivity. It occurs that this  invariant is maximized by the complete graph and minimized by the star. In the next natural step, we determine the maximizers and the minimizers within the class of $\mathcal{C}_{even}$-graphs excluding complete graphs and stars. 

A \textit{clique} in a graph is a set of vertices that are all adjacent to each other. The size of a maximum clique of a graph $G$ is known as a \textit{clique number}, denoted by $\omega(G)$. We give an explicit formula for the clique number of  $G\in\mathcal{C}_{even}$, and determine whether $\omega(G)$ is less than, equal to, or greater than $a(G)$; it appears that the last number $\alpha_k$ of the corresponding generating sequence is sufficient to establish this comparison.

Concerning related works, cographs have received a great deal of attention in the last six decades. Some notable results that are related to our results are obtained in \cite{Merris3} (where Merris proved that the Laplacian spectrum of every cograph consists entirely of integers), \cite{Tura6} (where Lazzarin et al.~proved that no two non-isomorphic equivalent cographs share the same Laplacian spectrum), \cite{Tara} (where Abrishami proved that for every non-complete cograph the algebraic connectivity and the vertex connectivity are equal), and \cite{Tura 3,Tura 4,Stanic,Jacobs 2,Trevisan 1,Royle} (where different authors have established many results concerning spectral properties of cographs and related graphs). Many results concerning lower and upper bounds for the algebraic connectivity can be found in \cite[Sections~6.6--6.9]{ifge}. Since, in case of cographs, this invariant coincides with the vertex connectivity, our results also relate the results concerning the bounds for the latter invariant, and some of them can be found in~\cite{Harary,Kirkland,Lucas}. In a classical paper \cite{Karp} Karp proved the NP-completeness of
21 combinatorial problems, one of them is a computing the maximal clique. Since then, this problem has been considered for many graph classes, and some results can be found in \cite{Godsil,Pecher} and references therein.

 Section \ref{sec:back} contains data about $\mathcal{C}$-graphs and some preliminary results concerning their Laplacian matrix. In Section \ref{sec:eande} we deal with the Laplacian eigenvalues and the corresponding eigenspaces. A range for the number of distinct Laplacian eigenvalues is given in Section~\ref{sec:diste}. The graphs $G\in \mathcal{C}_{even}$ that maximize or minimize $a(G)$ are considered in Section~\ref{sec:conn}. Section~\ref{sec:clique} is reserved for the clique number of a $\mathcal{C}_{even}$-graph.

\section{On $\mathcal{C}$-graphs}\label{sec:back}

 By a computer search, we found more than 1000 $\mathcal{C}$-graphs with 12 vertices and more than 8.388.600 $\mathcal{C}$-graphs  with 25 vertices. Observe that  $C(n-1,1)$ is the complete graph $K_n$, while $C(p-1,1,q)$ is the complete bipartite graph $K_{p,q}$. The following cographs are also  categorised as $\mathcal{C}$-graphs.
\begin{itemize}
	\item
	A \textit{split graph} is a graph whose vertex set admits a partition into a clique and a co-clique. A complete split graph, studied in \cite{KCD1}, is a split graph in which every vertex of the co-clique  is adjacent to every vertex in the clique. We observe that every complete split graph is a $\mathcal{C}$-graph represented by  $C(\alpha_1,\alpha_2)$.  However, the class of $\mathcal{C}$-graphs does not include all split graphs.
	\item
	An \textit{antiregular graph} is a connected graph whose degree sequence has only two repeated entries. Its representation is either $C(1,1,\ldots,1)$ or $C(1,2,1,1,\ldots,1)$.
	
	\item A \textit{chordal graph} is a graph without an induced subgraph isomorphic to the cycle $C_i,~ i\geq4$. Therefore, a cograph is a chordal graph if and only if it is $C_4$-free. Hence, the $\mathcal{C}$-graph $C(\alpha_1,\alpha_2,\ldots,\alpha_{k})$ is chordal whenever $\alpha_{2i}>1$ holds for at most one $i$, where $1\leq i\leq \frac{k}{2}$. We note in passing that a chordal cograph is also known as a quasi-threshold graph.  
\end{itemize}

There is no inclusion between the class $\mathcal{C}$ and the class of threshold graphs. However, concerning binary representations of threshold graphs given in \cite{Aguilar2,Bapat}, one may deduce that for a fixed number of vertices, the number of $\mathcal{C}$-graphs is never less than the number of threshold graphs. We shall skip the details and note that the same holds in comparison to the classes of chain graphs and complete multipartite graphs.

We recall that the eigenvalues of the Laplacian matrix are non-negative, zero is one of them and its multiplicity is equal to the number of components of a graph~\cite[Subsection~1.2.2]{ifge}. We proceed with a particular blocking of the Laplacian matrix. Although our focus is on $\mathcal{C}_{even}$-graphs, the following setting remain valid for $\mathcal{C}$-graphs. Evidently, a generating sequence $(\alpha_1, \alpha_2, \ldots, \alpha_k)$ of a $\mathcal{C}$-graph provides a partition of its vertex set. Moreover, vertices belonging to the same part share the same degree. In this context, we consider the corresponding equitable partition $\pi=\{\pi_{\alpha_1}, \pi_{\alpha_2}, \dots,\pi_{\alpha_k}\}$ such that $|\pi_{\alpha_i}|=\alpha_i$, for $1\leq i\leq k$. Let $d_{\alpha_i}$ denote the degree of a vertex in $\pi_{\alpha_i}$. Then we deduce that

$$d_{\alpha_i}=\begin{cases}
	\alpha_i-1 +\sum_{j~\text{even},~j \geq i+1}\alpha_j&\text{if}~i~\text{is odd},\\  
	\sum_{j=1}^{i-1}\alpha_j +\sum_{\ell~\text{even},~ \ell\geq i+2}\alpha_\ell&\text{if}~i~\text{is even and}~i<k,  \\
	\sum_{j=1}^{i-1}\alpha_j, & \text{if}~i=k.
\end{cases}$$

Accordingly, the Laplacian matrix of  $C(\alpha_1,\alpha_2,\dots,\alpha_k)$ admits the following blocking
\begin{equation}
\label{Lap_eq1}
    L=\begin{bmatrix}
	[(d_{\alpha_1}+1)I-J]&-J &O &-J& \dots & O&-J \\
	-J&d_{\alpha_2}I &O &-J&\dots &O&-J\\
	O&O&[(d_{\alpha_3}+1)I-J]&-J & \dots &O&-J\\
	-J&-J&-J&d_{\alpha_4}I&\dots&O&-J\\
	& & & & \ddots \\
	O&O&O&O&\dots&[(d_{\alpha_{k-1}}+1)I-J]&-J \\
	-J&-J&-J&-J& \dots &-J&d_{\alpha_k}I
\end{bmatrix},
\end{equation}
where $I$ and $J$ denote the identity matrix and the all-1 matrix, respectively.


Consequently the quotient matrix of $L$,  that correspond to  $\pi$, is the $k\times k$  matrix given by 
$$Q_L= \begin{bmatrix}
	[d_{\alpha_1}-(\alpha_1-1)]&-\alpha_2 &0 &-\alpha_4& \dots &0&- \alpha_k \\
	-\alpha_1&d_{\alpha_2}&0 &-\alpha_4& \dots &0&-\alpha_k\\
	0 &0&[d_{\alpha_3}-(\alpha_3-1)]&-\alpha_4 & \dots &0&-\alpha_k\\
	-\alpha_1&-\alpha_2&-\alpha_3&d_{\alpha_4}&\dots&0&-\alpha_k\\
	& & & & \ddots \\
	0&0&0&0&\dots&\alpha_k&-\alpha_k \\
	-\alpha_1&-\alpha_2&-\alpha_3&-\alpha_4& \dots &-\alpha_{k-1}&d_{\alpha_k}
\end{bmatrix}.$$

If $\lambda$ is an eigenvalue of $Q_L$, let $P=[p_{ij}]$ denote the $n\times k$ characteristic matrix for the equitable partition $\pi$, i.e., the $(i,j)$-th entry of $P$ is 
$$p_{ij}=\begin{cases}
	1 &\text{ if }i\in \pi_j,\\ 0 &\text{ otherwise}.
\end{cases}$$
From $LP=PQ_L$, we obtain $L(PX)=\lambda (PX)$, which  implies that every eigenvalue of $Q_L$ is also an eigenvalue of~$L$. 


\section{Eigenvalues and eigenvectors}\label{sec:eande}
We compute the  eigenvalues and the eigenvectors of the Laplacian matrix of a $\mathcal{C}$-graph in terms of a generating sequence. Assume that the eigenvalues of the quotient matrix $Q_L$ are arranged in non-decreasing order as follows
\begin{equation}
	\label{Lap_eq2}
	0=\lambda_1 \leq \lambda_2 \leq \lambda_3 \leq \dots \leq \lambda_k.
\end{equation}

\begin{theorem}
	\label{Lap_th1}
	The  eigenvalues of the quotient matrix $Q_L$ of a $\mathcal{C}_{even}$-graph $C(\alpha_1,\alpha_2,\dots,\alpha_k)$, where $k\geq4$, are $\lambda_1=0$,  $\lambda_k=n$ and
	$$\lambda_i=\begin{cases}
		\lambda_{i-1}+\alpha_{k-2(i-2)}&\text{for }2\leq i\leq \frac{k}{2},\\
		\lambda_{i+1}-\alpha_{2i-(k-1)}&\text{for }k-1\geq i\geq \frac{k}{2}+1.
	\end{cases}$$
\end{theorem}

\begin{proof}
	The all-1 vector $\mathbf{j}$ is associated with $\lambda_1=0$. We now construct the eigenvectors corresponding to the next $\frac{k}{2}-1$ smallest eigenvalues as follows: 
	{$$\mathbf{x}_i(j)=\begin{cases}
		1&\text{ if }1\leq j\leq k-2i+2,\\
		-\dfrac{\sum_{\ell=1}^{k-2i+2}\alpha_\ell}{\alpha_{k-2i+3}}&\text{ if } j=k-2i+3,\\
		0&\text{ otherwise.}
	\end{cases}$$
Indeed, for $i=2$ we have
$$\mathbf{x}_2=\begin{bmatrix}
    1 &1& 1 &\cdots&1&-\dfrac{\sum_{\ell=1}^{k-2}\alpha_\ell}{\alpha_{k-1}}& 0
    \end{bmatrix}^\intercal, \text{ along with }Q_L\mathbf{x}_2=\alpha_k\mathbf{x}_2,
$$
which implies that $\alpha_k$ is an eigenvalue of $Q_L$.\\
Similarly, for $i=3$,
$$\mathbf{x}_3=\begin{bmatrix}
    1 &1& 1 &\cdots&1&-\dfrac{\sum_{\ell=1}^{k-4}\alpha_\ell}{\alpha_{k-3}}& 0&0&0
    \end{bmatrix}^\intercal, \text{ and }Q_L\mathbf{x}_3=(\alpha_k+\alpha_{k-2})\mathbf{x}_3,
$$
which implies that $\alpha_k+\alpha_{k-2}$ is an eigenvalue of $Q_L$.
In general, for $2\leq i\leq\frac{k}{2}$, we obtain
	$$Q_L\mathbf{x}_i=(\alpha_k+\alpha_{k-2}+\cdots+\alpha_{k-2(i-2)})\mathbf{x}_i.$$
So, $\alpha_k+\alpha_{k-2}+\cdots+\alpha_{k-2(i-2)}$ is an eigenvalue of $Q_L$. This establishes the recurrence relation $\lambda_i=\lambda_{i-1}+\alpha_{k-2(i-2)}.$}\medskip
	
	For the remaining eigenvalues, we define  vectors 
	{$$\mathbf{x}_{\frac{k}{2}+i}(j)=\begin{cases}
		1&\text{ if }1\leq j\leq 2i-1,\\
		-\dfrac{\sum_{\ell=1}^{2i-1}\alpha_\ell}{\alpha_{2i}}&\text{ if } j=2i,\\
		0&\text{ otherwise.}
	\end{cases}.$$
 Now, for $i=\frac{k}{2}$, we obtain
 $$\mathbf{x}_k=\begin{bmatrix}
    1 &1& 1 &\cdots&1&-\dfrac{\sum_{\ell=1}^{k-1}\alpha_\ell}{\alpha_{k}}
    \end{bmatrix}^\intercal, \text{ and }Q_L\mathbf{x}_k=\Big(\sum_{\ell=1}^k\alpha_\ell\Big)\mathbf{x}_k=n\mathbf{x}_k,$$
    Therefore, the largest eigenvalue of $Q_L$ is $n$.
    In general, for $1\leq i\leq\frac{k}{2}$, it holds
	$$Q_L\mathbf{x}_{\frac{k}{2}+i}=\Big(\sum_{\ell=1}^{k/2}\alpha_{2\ell}+\sum_{m=1}^i\alpha_{2m-1}\Big) \mathbf{x}_{\frac{k}{2}+i},$$}
	which concludes the proof.
\end{proof}
 The following theorem gives the remaining eigenvalues of $L$.
\begin{theorem}
	
	\label{Lap_th2}
	The remaining $n-k$ eigenvalues of $L$ are 
	$d_{\alpha_{2i}}$ with multiplicity $\alpha_{2i}-1$  and $d_{\alpha_{2i-1}}+1$ with multiplicity $\alpha_{2i-1}-1$, for $i\leq \frac{k}{2}$.
\end{theorem}
\begin{proof}For $\ell>1$, let $\{E_j^\ell\}$ denote the set of orthogonal of $\ell-1$  row-vectors in $\mathbb{R}^\ell$ defined by
	$$E_j^\ell=\textbf{\emph{e}}_1(\ell)+\textbf{\emph{e}}_2(\ell)+\cdots+\textbf{\emph{e}}_{j}(\ell)-j\textbf{\emph{e}}_{j+1}(\ell)\ \text{for all}~j~\text{such that}~ 1\leq j\leq \ell-1,$$
	where $\textbf{\emph{e}}_{j}(\ell)$ is the $j$th row-vector of the canonical basis of $\mathbb{R}^\ell$. \\
	
	Now, for every $\alpha_i\geq 2$, we define
	$$\mathbf{x}_j^{\alpha_i}=[\mathbf{0}_{\alpha_1}\ \mathbf{0}_{\alpha_2}\ \cdots\ \mathbf{0}_{\alpha_{i-1}}\ E_j^{\alpha_i} \ \mathbf{0}_{\alpha_{i+1}}\ \cdots\ \mathbf{0}_{\alpha_k}]^\intercal,\ \  1\leq j \leq \alpha_i-1,\,1\leq i\leq k,$$
	where the  $\mathbf{0}_r$ denotes the all-0 row-vector in $\mathbb{R}^r$.\\	
{Then, for $\alpha_i\geq2$ and $1\leq s \neq t\leq\alpha_i-1$, we have
$$\big(\mathbf{x}_s^{\alpha_i}\big)^\intercal\mathbf{x}_t^{\alpha_i}=E_s^{\alpha_i}\big(E_t^{\alpha_i}\big)^\intercal=0.
$$
Therefore, the set $\{\mathbf{x}_1^{\alpha_i},\mathbf{x}_2^{\alpha_i},\ldots, \mathbf{x}_{\alpha_{2i-1}}^{\alpha_i}\}$ is orthogonal for all $\alpha_i\geq2$. Observe that, in one hand, by~\eqref{Lap_eq1}  the Laplacian  $L$ is a $k\times k$ block matrix whose non-diagonal blocks are constant matrices, while on the other hand, the entry-sum  of  $E_j^{\alpha_i}$ is 0 whenever $\alpha_i\geq2, 1\leq j\leq \alpha_i-1$. Thus, if $ 1\leq i\leq k$, for each $\alpha_{2i}\geq 2$, we obtain  
\begin{align*}
    L\mathbf{x}_j^{\alpha_{2i}}&=[\mathbf{0}_{\alpha_1}\ \mathbf{0}_{\alpha_2}\ \cdots\ \mathbf{0}_{\alpha_{i-1}}\ d_{\alpha_{2i}}E_j^{\alpha_i} \ \mathbf{0}_{\alpha_{i+1}}\ \cdots\ \mathbf{0}_{\alpha_k}]^\intercal=d_{\alpha_{2i}}{\mathbf{x}_j^{\alpha_{2i}}},
\end{align*}
 for all $1\leq j\leq \alpha_{2i}-1.$}
 
	Similarly, for $\alpha_{2i-1}\geq 2$, the vectors $ \mathbf{x}_j^{\alpha_{2i-1}}$ satisfy $$L\mathbf{x}_j^{\alpha_{2i-1}}=(d_{\alpha_{2i-1}}+1){\mathbf{x}_j^{\alpha_{2i-1}}},\ \ \text{ for } 1\leq j\leq \alpha_{2i}-1,$$ 
and this completes the proof.
\end{proof}

We provide more details in a particular case $k=2$. Despite it is simple, this case is illustrative since it computes the eigenvectors according to the previous theorems. Also, it will be used in the sequel.

\begin{example}
\label{remark0}
For $k=2$, the corresponding $\mathcal{C}_{even}$-graph is the complete split graph $C(\alpha_1,\alpha_2)$. The quotient matrix $Q_L$ is
$$Q_L= \begin{bmatrix}
	[d_{\alpha_1}-(\alpha_1-1)]&-\alpha_2 \\
	-\alpha_1&d_{\alpha_2}
\end{bmatrix}=\begin{bmatrix}
	\alpha_2&-\alpha_2 \\
	-\alpha_1&\alpha_1
\end{bmatrix}
$$
Its eigenvalues are $\lambda_1=0$ an  $\lambda_2=\alpha_1+\alpha_2$.  The remaining two eigenvalues of $L$ are $d_{\alpha_2}$ with multiplicity $(\alpha_2-1)$ and $d_{\alpha_1}+1$ with multiplicity $(\alpha_1-1)$. 
In the particular case $\alpha_1=\alpha_2=1$ we deal with a 2-vertex graph with Laplacian eigenvalues 0 and 2. For $\alpha_1,\alpha_2 \geq 2$, let $\mathbf{x}_i$,  $1\leq i\leq (\alpha_2-1)$, and $\mathbf{y}_j$,  $1\leq  j\leq (\alpha_1-1)$, be the eigenvectors associated with $d_{\alpha_2}$ and $d_{\alpha_1}+1$, respectively. Then, 
$$\mathbf{x}_1=[\underbrace{0 ~ 0~\cdots ~0}_\text{$\alpha_1$}  ~1~-1~0~0~\cdots~0]^\intercal, $$

$$\mathbf{x}_2=[\underbrace{0 ~0~ \cdots ~0}_\text{$\alpha_1$}~1~1~-2~ 0~0~ \cdots ~0 ]^\intercal,$$
$$\mathbf{x}_3=[\underbrace{0 ~0~ \cdots ~0}_\text{$\alpha_1$}~1~1~1~-3~ 0~0~ \cdots ~0 ]^\intercal,$$
$~~~~~~~~~~~~~~~~~~~~~~~~~~~~~~~~~~~~~~~~~~~~~~~~~~~~~~~~~~~~~~~~~~~~~~~~~~~\vdots$

$$\mathbf{x}_{\alpha_2-1}=[\underbrace{0 ~0~ \cdots ~0}_\text{$\alpha_1$}~\underbrace{1~1~\cdots~1}_{\text{$\alpha_2-1$}}~-(\alpha_2-1) ]^\intercal.$$
Similarly,

$$\mathbf{y}_1=[1~-1~ 0~0~\cdots ~0 ~\underbrace{0~0~\cdots~0}_\text{$\alpha_2$}]^\intercal, $$

$$\mathbf{y}_2=[1~1~-2~ 0~0~\cdots ~0 ~\underbrace{0~0~\cdots~0}_\text{$\alpha_2$}]^\intercal, $$

$$\mathbf{y}_3=[1~1~1~-3~ 0~0~\cdots ~0 ~\underbrace{0~0~\cdots~0}_\text{$\alpha_2$}]^\intercal, $$

$~~~~~~~~~~~~~~~~~~~~~~~~~~~~~~~~~~~~~~~~~~~~~~~~~~~~~~~~~~~~~~~~~~~~~~~~~~~\vdots$

$$\mathbf{y}_{\alpha_1-1}=[\underbrace{1~1~\cdots ~1}_\text{$\alpha_1-1$}~-(\alpha_1-1)~ \underbrace{0~0~\cdots~0}_\text{$\alpha_2$}]^\intercal. $$
\end{example}


We now provide two straightforward consequences of the previous results. The first one follows directly.
\begin{corollary}\label{cor:simple}
	The matrix $Q_L$ has simple eigenvalues.
\end{corollary}


\begin{corollary}
	\label{Lap_cor3}
	The algebraic connectivity of a $\mathcal{C}_{even}$-graph $G\cong C(\alpha_1,\alpha_2,\dots,\alpha_k)$ is 
	$$a(G)=\begin{cases}
		\alpha_1&\text{if }k=2~\text{and }~\alpha_2\neq1,\\
                  \alpha_1+\alpha_2 & \text{if }k=2~\text{and}~\alpha_2=1,\\
		\min\{\alpha_k,n-\alpha_k\}&\text{if } k\geq 4.
	\end{cases}
	$$
\end{corollary}
\begin{proof}
	If $k=2$, then  $G$ is a complete  split graph $C(\alpha_1, \alpha_2)$ with Laplacian eigenvalues  $ (\alpha_1+\alpha_2)^{\alpha_1}, \alpha_1^{\alpha_2-1}$ and $0$,
		where exponents stand for the multiplicities. Clearly, the second smallest eigenvalue is $\alpha_1$ when $\alpha_2\neq1$, and $\alpha_1+\alpha_2$ when $\alpha_2=1$.
		
If $k\geq 4$, then, by Theorems~\ref{Lap_th1} and~\ref{Lap_th2}, the second smallest eigenvalue of  $G$  is either $\alpha_k$, or $d_{\alpha_{2i}}$ for some $\alpha_{2i}\geq2$, or $d_{\alpha_{2i-1}}+1$ for some $\alpha_{2i-1}\geq2$. Here we observe that, $d_{\alpha_{2i}}>\alpha_{k}$ for all $1\leq i\leq k-1$ and  $d_{\alpha_{2i-1}}+1>\alpha_k$ for all $1\leq i\leq k$.
Therefore, $a(G)$ is either $\alpha_k$ or $d_{\alpha_{k}}=n-\alpha_k$ with $\alpha_k\neq1$. Now, $d_{\alpha_k}\geq\alpha_k$ gives $\alpha_k\geq  \frac{n}{2}>1$, so in this case $n-\alpha_k$ occurs in the spectrum of $L$. Therefore $a(G)=\min\{\alpha_k,n-\alpha_k\}$, as desired.
\end{proof}

We conclude the section with another  example.

\begin{example}
	Let us consider the $\mathcal{C}_{even}$-graph $C(8,3,4,2,1,5,6,3,7,9)$ with $48$ vertices. The quotient matrix $Q_L$ is the $10 \times 10$ matrix given by
	$$Q_L=\begin{bmatrix}
		22&-3 &0 &-2&0 &-5&0&-3 & 0&-9 \\
		-8&27 &0 &-2&0&-5&0&-3&0&-9\\
		0&0&19&-2 &0 &-5& 0&-3&0&-9\\
		-8&-3&-4&32&0&-5&0&-3&0&-9\\
		0&0&0&0&17&-5&0&-3&0&-9\\
		-8&-3&-4&-2&-1&30&0&-3&0&-9\\
		0&0 &0 &0 &0 &0 &12&-3&0&-9 \\
		-8&-3&-4&-2&-1&-5&-6&38&0&-9\\
		0&0&0&0&0&0&0&0&9&-9 \\
		-8&-3&-4&-2&-1&-5&-6&-3&-7&39
	\end{bmatrix}.$$
	By Theorem \ref{Lap_th1}, the eigenvalues of $Q_L$ are $48, 41, 35, 34, 30, 19, 17, 12, 9, 0$, while the corresponding eigenvectors are (a subscript denotes the eigenvalue)
	$$\mathbf{x}_{48}=[1\ 1\ 1\ 1\ 1\ 1\ 1\ 1\ 1 -\frac{13}{3}]^\intercal  ,~~\mathbf{x}_{0}=[1\ 1\ 1\ 1\ 1\ 1\ 1\ 1\ 1\ 1]^\intercal$$
	$$\mathbf{x}_{41}=[1\ 1\ 1\ 1\ 1\ 1\ 1\ -\frac{29}{3}\ 0\ 0]^\intercal ,~~\mathbf{x}_{9}=[1\ 1\ 1\ 1\ 1\ 1\ 1\ 1\ -\frac{32}{7}\ 0]^\intercal$$
	$$\mathbf{x}_{35}=[1\ 1\ 1\ 1\ 1, -\frac{18}{5} 0\ 0\ 0\ 0]^\intercal ,~~\mathbf{x}_{12}=[1\ 1\ 1\ 1\ 1\ 1\ -\frac{23}{6}\ 0\ 0\ 0]^\intercal$$
	$$\mathbf{x}_{34}=[1\ 1\ 1\ -\frac{15}{2}\ 0\ 0\ 0\ 0\ 0\ 0]^\intercal ,~~\mathbf{x}_{17}=[1\ 1\ 1\ 1\ -17\ 0\ 0\ 0\ 0\ 0]^\intercal $$
	$$ \mathbf{x}_{30}=[1\ -\frac{8}{3}\ 0\ 0\ 0\ 0\ 0\ 0\ 0\ 0]^\intercal ,~~\mathbf{x}_{19}=[1\ 1\ -\frac{11}{4}\ 0\ 0\ 0\ 0\ 0\ 0\ 0]^\intercal.$$
	
	Using Theorem \ref{Lap_th2}, we compute the remaining eigenvalues and their multiplicities: $27^2,32,30^4,38^2,39^8,30^7,23^3,$ $18^5,16^6 $. The corresponding eigenvectors are computed as in Example~\ref{remark0}.
\end{example}

\section{Number of distinct eigenvalues}\label{sec:diste}

This section is devoted to the number of distinct eigenvalues, denoted by $s(G)$, of the Laplacian matrix of a $\mathcal{C}_{even}$-graph  $G$. We start with the following theorem.
\begin{theorem}
	\label{Lap_th3}
	For a $\mathcal{C}_{even}$-graph $C(\alpha_1,\alpha_2,\ldots,\alpha_k)$, 
	\begin{equation}
		\label{Lap_eq3}
		k\leq s(G)\leq 2k-1.
	\end{equation}
\end{theorem}
\begin{proof}
	By Corollary \ref{cor:simple}, the quotient matrix $Q_L$ has exactly $k$ distinct eigenvalues. Thus, it follows that $s(G)\geq k$. 
	
	Theorem \ref{Lap_th2} says that $d_{\alpha_{2i}}$ and $d_{\alpha_{2i-1}}+1$ are the eigenvalues for all $1\leq i\leq \frac{k}{2}$; this gives at most $k$ distinct eigenvalues in the spectrum of $L$. In addition, $Q_L$ has $k$ distinct eigenvalues. Together, we have at most $2k$ distinct eigenvalues. However, we observe that  $(\frac{k}{2}+1)$th eigenvalue of $Q_L$ is always equal to $d_{\alpha_1}+1$. Thus we obtain $s(G)\leq 2k-1$, and this proves the right-hand side of~\eqref{Lap_eq3}.
\end{proof}

In the next two remarks, we will see that the obtained bounds for $s(G)$ are sharp. 

\begin{remark}
	\label{Lap_rem1}
	We observe that the lower bound of (\ref{Lap_eq3}) is attained in each of the following cases:
	\begin{enumerate}
		\item
		Let $G\cong C(\alpha_1,1,1,\ldots,1)$. Here, $Q_L$ has $k$ distinct eigenvalues and $d_{\alpha_1}+1$ is an additional eigenvalue of $L$. However, $d_{\alpha_1}+1$ belongs to the spectrum of $Q_L$, as mentioned in the previous proof. Therefore,  $G$ has exactly $k$ distinct eigenvalues.
		\item Let $G\cong C(1,1,\ldots,1, p, 1)$, with  $2\leq p\leq\frac{k}{2}-2$. Then, by Theorem \ref{Lap_th2}, $d_{\alpha_{k-1}}+1$ is an eigenvalue of $L$, but it equals the $(p+2)$th eigenvalue of $Q_L$. As before, $G$ has exactly $k$ distinct eigenvalues.
		\item Let $G\cong C(\alpha_1,\alpha_2,\alpha_3,\alpha_4)$. First, note that the eigenvalues of $Q_L$ are $0$, $\alpha_4$, $\alpha_1+\alpha_2+\alpha_4=d_{\alpha_1}+1$ and $n$. Thus if $G$ has exactly $4$ distinct eigenvalues, then any eigenvalue of $L$ obtained by Theorem~\ref{Lap_th2} must be equal to either  $\alpha_4$ or $\alpha_1+\alpha_2+\alpha_4$. In this case, $G$ is one of the following (in all cases, $\alpha\geq 1$):
		\begin{enumerate}
			\item
			$G\cong C(\alpha,1,1,1)$
			\item
			$G\cong C(\alpha,1,1+\alpha,1)$
			\item
			$G\cong C(\alpha,1,1,2+\alpha)$
			\item
			$G\cong C(\alpha,1,1+\alpha,1+\alpha)$
			\item
			$G\cong C(\alpha,1,1+\alpha,2+2\alpha)$
		\end{enumerate}
		This item also characterizes all $\mathcal{C}_{even}$-graphs with exactly four distinct eigenvalues.
	\end{enumerate}
\end{remark}

\begin{remark}
	\label{Lap_rem2}
	Here are some $\mathcal{C}_{even}$-graphs attaining the upper bound of (\ref{Lap_eq3}):
	\begin{enumerate}
		\item
		Let $G$ ($\not\cong K_n$) be the complete  split graph $C(\alpha_1, \alpha_2)$. The eigenvalues of $L$ are $$ (\alpha_1+\alpha_2)^{\alpha_1}, \alpha_1^{\alpha_2-1}, 0,$$
		 and so the upper bound of (\ref{Lap_eq3}) is attained. Observe also that this item characterizes all  $\mathcal{C}_{even}$-graphs with exactly three distinct eigenvalues.
		\item
		Let $G\cong C(p,q,p+1,q+1)$, where $p\neq q$ and $q>1$.
		The eigenvalues of $L$ are $$2p+2q+2,\ (p+2q+1)^p,\ q+1,\ 0,\ (p+q+1)^{q-1},\ (2p+q+1)^q,\ (p+q+2)^p, $$
		along with the desired conclusion. 
		\item
		Similarly, for $G\cong C(i,j,r,i+1,j+1,r+1)$, where $j>1$, $i\neq r$ and $r+i\neq j$, 
		 by Theorems~\ref{Lap_th1} and~\ref{Lap_th2}, the  eigenvalues of $L$ are $2i+2j+2r+3,~2i+j+2r+2,~r+i+2,~r+1,~0,~(2i+r+2)^{j-1},~(i+j+2r+1)^i,~(2i+2j+r+2)^r,~(2i+j+r+2)^{i},~(2r+i+2)^{r-1},~(j+r+2)^j$. 
		 Hence, $s(G)=2k-1=11.$
	\end{enumerate}
\end{remark}
Next we consider a particular case of a constant sequence.

\begin{theorem}
	\label{Lap_th4}
	Let $G\cong C(\underbrace{p, p,\ldots, p}_k)$ be a $\mathcal{C}_{even}$-graph, where $p\neq1$. Then
	$s(G)=k+1.$
\end{theorem}
\begin{proof}
	Let $\lambda_1,\lambda_2,\ldots,\lambda_k$ be the eigenvalues of $Q_L$,  arranged as in~\eqref{Lap_eq2}. By Theorem \ref{Lap_th2}, $G$ has $\frac{k}{2}$ eigenvalues of the form $d_{\alpha_{2i}}$ and $\frac{k}{2}$ eigenvalues of the form $d_{\alpha_{2i-1}}+1$. We note the following overlapping between the eigenvalues: 
	$$d_{\alpha_2}=d_{\alpha_3}+1,$$
	$$d_{\alpha_4}=d_{\alpha_1}+1=\lambda_{\frac{k}{2}+1},$$
	$$d_{\alpha_{k-i}}=\lambda_{k-(\frac{i}{2}+1)} , ~\text{for}~i\in\{0,2,4,\ldots, k-4\},$$
	$$d_{\alpha_{k-j}}+1=\lambda_{k-(\frac{3-j}{2}+6)} , ~\text{for}~j\in\{1,3,5,\ldots, k-5\}.$$
	
	Therefore, the eigenvalues of $Q_L$ contribute $k$ to $s(G)$, and the eigenvalues of the form $d_{\alpha_{2i}}$ ($\neq d_{\alpha_2}$) and $d_{\alpha_{2i-1}}+1$ do not contribute anything extra to $s(G)$; however, $d_{\alpha_2}$ contributes one. Hence, $s(G)=k+1$.
\end{proof}

The following table contains a list of some $\mathcal{C}_{even}$-graphs with 24 vertices along with their distinct eigenvalues.

\begin{table}[h]
	\begin{center}
		\begin{tabular}{cccc}
			\hline
			$k$ & $G$& $s(G)$ & distinct eigenvalues\\ \hline
			2 & $C(21,3)$ & $3$ & 0, 21, 24\\
			\hline
			4 & $C(5,1,6,12)$ & 4 & 0, 12, 18, 24 \\
			\hline
			4 & $C(6,6,6,6)$ & 5  & 0, 6, 12, 18, 24\\
			\hline
			4 & $C(10,1,10,3)$ & 6 & 0, 3, 13, 14, 21, 24\\
			\hline
			6 & $C(4,4,4,4,4,4)$ & 7 & 0, 4, 8, 12, 16, 20, 24\\
			\hline
			4 & $C(4,7,5,8)$ & 7 & 0, 8, 12, 13, 16, 19, 24\\
			\hline
			8 & $C(17,1,1,1,1,1,1,1)$ & 8  &0, 1, 2, 3, 16, 21, 22, 23, 24\\
			\hline
			8 & $C(3,3,3,3,3,3,3,3)$ & 9 & 0, 3, 6, 9, 12, 15, 18, 21, 24\\
			\hline
			6 & $C(2,3,4,4,5,6)$ & 10 & 0, 6, 10, 11, 12, 14, 15, 18, 19, 24\\
			\hline
			6 & $C(5,2,3,4,2,8)$ & 11 &0, 8, 10, 12, 15, 16, 17, 18, 19, 22, 24\\
			\hline
		\end{tabular}
	\end{center}
	\caption{Distinct eigenvalues of some $\mathcal{C}_{even}$-graphs.}
	\label{Lap_table1}
\end{table}


\section{Connectivity}\label{sec:conn}

The algebraic connectivity $a(K_n)$ of a complete graph $K_n$ is $n$, and we know from \cite{MF1,ifge} that this graph maximizes the algebraic connectivity in the set of all graphs with $n$ vertices. The vertex connectivity $\kappa=\kappa(G)$ is maximized by the same graph \cite{MF1} and it equals $n-1$. 
The classical result of Fiedler~\cite{MF1} states that 
\begin{equation*}\label{eq:Fiedler}a(G)\leq\kappa(G)\leq \delta(G),\end{equation*} 
holds for every connected non-complete graph, where $\delta$ denotes the minimum vertex degree. Moreover, we have pointed out in the first section that, in case of cographs, the first equality is attained. We easily obtain cographs that minimize the algebraic connectivity.

\begin{lemma}
	\label{Lap_th5}
	For any connected cograph $G$, $a(G)$ is an integer and $a(G)\geq1.$ If $G$ is a star $K_{1, n-1}, n\geq 3$, then $a(G)=1$. 
\end{lemma}
\begin{proof}
	First, $a(G)$ is an integer since the Laplacian eigenvalues of $G$ are integral. Since $G$ is connected, it holds $a(G)\geq 1$. For a star with at least three vertices, we have $a(K_{1, n-1})= \kappa(K_{1, n-1})=1$, which concludes the proof. 
\end{proof}

In what follows we determine connected non-complete $\mathcal{C}_{even}$-graphs with a fixed number of vertices  that maximize the algebraic connectivity, and we also determine connected $\mathcal{C}_{even}$-graphs with a fixed number of vertices that are not stars and  minimize the algebraic connectivity.

\begin{theorem}
	Among all connected non-complete $\mathcal{C}_{even}$-graphs with  $n$ vertices, the graph $C(n-2,2)$  maximizes the algebraic connectivity.
\end{theorem}
\begin{proof}
	Let $G\cong C(\alpha_1,\alpha_2,\ldots, \alpha_k)$ be a connected non-complete $\mathcal{C}_{even}$-graph. First note that for $k=2$, the algebraic connectivity is maximized by $C(n-2,2)$, along with $a(C(n-2,2))=n-2$, see Example~\ref{remark0}. For $k\geq 4$, Corollary~\ref{Lap_cor3} gives $a(G)=\min\{\alpha_k, n-\alpha_k\}< n-2$, since $\alpha_i\geq 1$ for all $i$.
\end{proof}

\begin{theorem}
	Among all connected $\mathcal{C}_{even}$-graphs that are non-isomorphic to the star and have $n$ vertices, the graph $C(\alpha_1,\alpha_2,\ldots, \alpha_{k-2},\alpha_{k-1},1),\ k\geq4$, minimizes the algebraic connectivity.
\end{theorem}
\begin{proof} Let $G\cong C(\alpha_1,\alpha_2,\ldots, \alpha_k)$ be the graph under consideration, and set first $k\geq 4$. By Corollary~\ref{Lap_cor3}, we have $a(G)=\min\{\alpha_k, n-\alpha_k\}$, and the desired result  follows. It remains to show that $a(G)>1$ holds for $k=2$. Applying Corollary~\ref{Lap_cor3}, under the restrictions given in the formulation of this statement, we obtain the required inequality.
\end{proof}

Observe that the complete split graph $G$ that is not a star minimizes the algebraic connectivity if  and only if $G\cong{C(2, n-2)}$, which is a direct consequence of Corollary~\ref{Lap_cor3}.

%
%
%

\section{Clique number}\label{sec:clique}

In this section we compute the clique number of a $\mathcal{C}_{even}$-graph and establish a relationship with the algebraic connectivity.


\begin{theorem}
	\label{Lap_th6}
	Let $G=C(\alpha_1,\alpha_2,\ldots,\alpha_k)$ be a $\mathcal{C}_{even}$-graph. Then its clique number is 
	$$\omega(G)=\max_{1\leq i\leq\frac{k}{2}}\Big\{ \alpha_{2i-1}+\frac{k}{2}-i+1\Big\}.$$
\end{theorem}
\begin{proof}
	Consider the equitable partition $\pi$  of $G$ (defined in Section~\ref{sec:back}). The vertices of  $\pi_i$ form a clique if $i$ is odd, whereas, if $i$ is even then they form a co-clique (that is an edgeless graph). Further, every vertex of $\pi_{2i-1}$, $1\leq i\leq\frac{k}{2}$, is adjacent to every vertex of $\pi_{2j}$ whenever $j\geq i$. Thus, a clique in $G$ is formed by taking $\alpha_{2i-1}$  vertices of $\pi_{2i-1}$ together with one vertex from each of  $\pi_{2j}$, where $j\geq i$. Clearly, such a clique counts $\alpha_{2i-1} + \frac{k}{2}-i+1$ vertices. The maximum clique is obtained by taking the maximum over $i~(1\leq i\leq\frac{k}{2})$, which brings us to the desired result.
\end{proof}
The following corollaries are immediate applications of Theorem \ref{Lap_th6}.
\begin{corollary}\label{cor:cliq1}
	For the complete split graph $C(\alpha_1,\alpha_2)$, we have $\omega(C(\alpha_1,\alpha_2))=\alpha_1+1.$
\end{corollary}
\begin{corollary}
	For the antiregular graph $C(1,1,\ldots,1)$, we have   $\omega(C(1,1,\ldots,1))=\frac{k}{2}+1.$
\end{corollary}

We also emphasize the following result.

\begin{corollary}\label{cor:Omoj}For $k\geq 4$ and a $\mathcal{C}_{even}$-graph $G\cong C(\alpha_1, \alpha_2, \ldots, \alpha_k)$ with $n$ vertices, we have $\omega(G)\leq n-\alpha_k$, with equality if and only if $k=4$ and $\alpha_j=1$ for $2\leq j\leq 4$.
\end{corollary}

\begin{proof} For $1\leq i\leq \dfrac{k}{2}$, we have 
	\begin{equation}
		\label{eq:corO}
	n-\alpha_k=\sum_{j=1}^{k-1}\alpha_j\geq \alpha_{2i-1}+k-2\geq \alpha_{2i-1}+\frac{k}{2}-i+1,\end{equation}
	where the first inequality follows from $\alpha_j\geq 1$ for every $j$, and the second one follows from $k\geq 4$. Together with Theorem~\ref{Lap_th6}, this gives $\omega(G)\leq n-\alpha_k$. 
	
	If the equality holds, then we have equalities in \eqref{eq:corO}. The former one yields $\alpha_j=1$ for $j\neq 2i-1$. The latter one gives $k-2=\frac{k}{2}-i+1$, that is $k-6+2i=0$, which yields $k=4$ and $i=1$. Therefore, $G\cong C(\alpha_1, 1, 1 ,1)$. The opposite implication follows directly.
\end{proof}

In  case of a complete split graph, the clique number and the algebraic connectivity are computed easily, by employing Corollaries~\ref{Lap_cor3} and \ref{cor:cliq1}. The next result relates these invariants for the remaining $\mathcal{C}_{even}$-graphs.

\begin{theorem}
	Let $G\cong C(\alpha_1, \alpha_2, \ldots, \alpha_k)$ for $k\geq 4$. Then 
	$$\omega(G)\begin{cases}
		<a(G)&\text{if } \omega(G)<\alpha_k,\\ 
		=a(G)&\text{if } \omega(G)=\alpha_k,\\ 
		>a(G)&\text{if } \omega(G)>\alpha_k.
	\end{cases}$$	
\end{theorem}
\begin{proof} Assume that $\alpha_k<n-\alpha_k$. In this case, by Corollary~\ref{Lap_cor3}, we have $a(G)=\alpha_k$, which establishes the desired result. 
	
	For $\alpha_k\geq n-\alpha_k$, we have
	$$\omega(G)<n-\alpha_k=a(G)\leq \alpha_k,$$
	where the first inequality follows from Corollary~\ref{cor:Omoj}, and the remaining two follow from Corollary~\ref{Lap_cor3}. The last chain of inequalities gives the desired result.
\end{proof}

We proceed with particular cases that illustrate the result of the previous theorem.

\begin{corollary}
	For each of the following $\mathcal{C}_{even}$-graphs $G$, the inequality $\omega(G)>a(G)$ holds:
	\begin{enumerate}
		\item
		$G\cong C(\alpha_1, \alpha_2)$, {with $\alpha_2\geq2$}.
		\item
		$G\cong C(\underbrace{p, p, \ldots, p}_k)$, with $k\geq 4$.
		\item {$G\cong C(\underbrace{\alpha, \beta,\alpha, \beta, \ldots, \alpha , \beta}_{k\geq 4})$, with  $\alpha>\beta$}.
	\end{enumerate}
\end{corollary}

\begin{corollary}
	For  $G\cong C(\alpha,\alpha+1,\ldots , \alpha+k-1)$, {with $k\geq4$, we have $\omega(G)=a(G).$}
\end{corollary}

\begin{corollary}
	For  $G\cong C(\alpha,\alpha^2,\ldots, \alpha^k)$, with $k\geq4$ and $\alpha>1$, we have $\omega(G)<a(G).$
\end{corollary}

We conclude this section with a review of $\mathcal{C}_{even}$ graphs illustrating how the clique number and the algebraic connectivity may differ from one another. They are given in Table~\ref{Lap_table2}.

\begin{table}[h]
	\begin{center}
		\begin{tabular}{cccccc}
			\hline
			No. & $(n, k)$ & $G$& $\omega(G)$ & $a(G)$ & comparison \\ \hline
			1&$(57,2)$ & $C(24,33)$ & $25$ & $24$ & $\omega(G)>a(G)$\\
			\hline
			2&$(4,4)$ & $C(1,1,1,1)$ & $3$ & $1$  & $\omega(G)>a(G)$\\
			\hline
			3&$(14,6)$ & $C(5,1,1,1,1,5)$ & $8$ & $5$  & $\omega(G)>a(G)$ \\
			\hline
			4&$(231,6)$ & $C(32,59,26,19,66,29)$ & $65$ & $29$ & $\omega(G)>a(G)$ \\
			\hline
			5&$(35,4)$ & $C(6,13,8,8)$ & $9$ & $8$  & $\omega(G)>a(G)$\\
			\hline
			6&$(28,4)$ & $C(8,3,2,15)$ & $10$ & ${13}$ & $\omega(G)<a(G)$ \\
			\hline
			7&$(43,4)$ & $C(14,9,4,16)$ & $16$ & $16$  & $\omega(G)=a(G)$ \\
			\hline
			8&$(125,6)$ & $C(20,11,15,19,29,31)$ & $30$ & $31$ & $\omega(G)<a(G)$ \\
			\hline
			9&$(191,6)$ & $C(41,29,45,35,21,20)$ & $47$ & $20$   & $\omega(G)>a(G)$\\
			\hline
			10&$(221,6)$ & $C(35,20,31,40,45,50)$ & $46$ & $50$ & $\omega(G)<a(G)$ \\
			\hline
		\end{tabular}
	\end{center}
	\caption{A comparison between the clique number and the algebraic connectivity on some random $\mathcal{C}_{even}$-graphs}
	\label{Lap_table2}
\end{table}

\section{Acknowledgements}
The  research of Santanu Mandal is supported by the University Grants Commission of India under the beneficiary code BININ01569755. The research of Zoran Stani\' c is supported by the Fund of the Republic of Serbia; grant number 7749676:
Spectrally Constrained Signed Graphs with Applications in Coding Theory and Control Theory --
	SCSG-ctct.

\section{Statements and Declarations} \textbf{Competing Interests:} The authors made no mention of any potential conflicts of interest.

\end{document}